\documentclass[11pt]{amsart}
\usepackage{latexsym,amssymb,amsfonts,amsmath,graphicx,fullpage,url}
\usepackage[vcentermath,enableskew]{youngtab}
\usepackage{color}
\usepackage{cite}
\usepackage[all]{xy}
\usepackage{tikz,bbm}
\usepackage{verbatim}
\usepackage{tikz-cd}
\usepackage{latexsym,amssymb,amsfonts,amsmath,graphicx,fullpage,url, cite}
\usepackage[vcentermath,enableskew]{youngtab}
\usepackage{color}
\usepackage{cite}
\usepackage[all]{xy}
\usepackage{verbatim}
\usepackage{tikz,bbm}
\usepackage{caption}
\usepackage{subcaption}
\usepackage{hyperref}
\usepackage{bm}
\usepackage{colonequals}
\usepackage{mathtools}

\newtheorem{theorem}[subsection]{Theorem}
\newtheorem{proposition}[subsection]{Proposition}

\newtheorem{lemma}[subsection]{Lemma}
\newtheorem{corollary}[subsection]{Corollary}
\theoremstyle{definition}
\newtheorem{definition}[subsection]{Definition}
\newtheorem{example}[subsection]{Example}

\theoremstyle{remark}

\numberwithin{equation}{section}
\def\F{\mathcal{F}}
\def\vol{{\rm vol}}

\def\F{{\mathcal{F}}}

\def\R{\mathbb{R}}

\def\aff{\mathop{\mathrm{aff}}}
 
\def\m{{\bf m}}
\newcommand{\be}{\begin{equation}}
\newcommand{\ee}{\end{equation}}

\newcommand{\B}{\bigg}

\newcommand{\indeg}{\it indeg}
\newcommand{\outdeg}{\it outdeg}
\newcommand{\inn}{{\bf in}}
\newcommand{\lout}{{\bf out}}
\newcommand{\lin}{{\bf in}}
\newcommand{\Gc}{{G({\bf c})}}

\newcommand{\bd}{\begin{definition}}
\newcommand{\ed}{\end{definition}}
\newcommand{\bt}{\begin{theorem}}
\newcommand{\et}{\end{theorem}}
\newcommand{\bl}{\begin{lemma}}
\newcommand{\el}{\end{lemma}}
\newcommand{\bp}{\begin{proposition}}
\newcommand{\ep}{\end{proposition}}
\newcommand{\bc}{\begin{corollary}}
\newcommand{\ec}{\end{corollary}}

\newtheorem*{theorem1*}{Theorem \ref{thm:main}}

\newcommand{\old}[1]{}

\def\ee{{\bf e}}
  
   \def\vol{{\rm vol}}

\def\m{{\bf m}}
 \def\f_H{{\bf w}}
 \def\f{{\bf f}}
 \def\a{{\bf a}}

\def\R{\mathbb{R}}

\def\Z{\mathbb{Z}}
\def\g{{\bf a}}

 \def\F{\mathcal{F}}
\def\O{\mathcal{O}}
\def\I{\mathcal{I}}

\def\ee{{\bf e}}
  
   \def\vol{{\rm vol}}

\def\aa{{\bf a}}
 \def\f_H{{\bf w}}
 \def\f{{\bf f}}
 \def\a{{\bf a}}
\def\R{\mathbb{R}}

\def\Z{\mathbb{Z}}
\def\g{{\bf a}}

 \def\F{\mathcal{F}}
\def\O{\mathcal{O}}
\def\I{\mathcal{I}}

\def\ee{{\bf e}}
  
   \def\vol{{\rm vol}}

\def\om{{\rm out}}

\def\aa{{\bf a}}
 \def\f_H{{\bf w}}
 \def\f{{\bf f}}
 \def\a{{\bf a}}

\def\R{\mathbb{R}}

\def\Z{\mathbb{Z}}
\def\g{{\bf a}}

 \def\F{\mathcal{F}}
\def\O{\mathcal{O}}
\def\I{\mathcal{I}}

\newcommand{\eee}{\end{equation}}

\newcommand\multiset[2]%
{\mathchoice{\left(\kern-0.5em{\binom{#1}{#2}}\kern-0.5em\right)}
            {\bigl(\kern-0.3em{\binom{#1}{#2}}\kern-0.3em\bigr)}
            {\bigl(\kern-0.3em{\binom{#1}{#2}}\kern-0.3em\bigr)}
            {\bigl(\kern-0.3em{\binom{#1}{#2}}\kern-0.3em\bigr)}}

\def\R{\mathbb{R}}

\title{Counting  integer points of flow polytopes}

\author{Kabir Kapoor}
\author{Karola M\'esz\'aros}
\author{Linus Setiabrata}
\thanks{Work of M\'esz\'aros partially supported by  NSF Grant DMS 1501059  and   CAREER NSF Grant DMS 1847284 as well as by a von Neumann Fellowship at the IAS   funded by the Fund for Mathematics and the Friends of the Institute for Advanced Study.}

\address{Department of Mathematics, Cornell University, Ithaca, NY 14853}
\address{Department of Mathematics, Cornell University, Ithaca, NY 14853 and School of Mathematics, Institute for Advanced Study, Princeton, NJ 08540}
\address{Department of Mathematics, Cornell University, Ithaca, NY 14853}
 
\email{ksk86@cornell.edu, karola@math.cornell.edu, ls823@cornell.edu}

\date{\today}

\begin{document}
\maketitle

\begin{abstract}
The Baldoni--Vergne volume and Ehrhart polynomial formulas for flow polytopes are significant in at least two ways. On one hand, these formulas are in terms of Kostant partition functions, connecting flow polytopes to this classical vector partition function fundamental in representation theory. On the other hand the Ehrhart polynomials can be read off from the volume functions of flow polytopes. The latter is remarkable since the leading term of the Ehrhart polynomial of an integer polytope is its volume! Baldoni and Vergne proved these formulas via residues. To reveal the geometry of these formulas, the second author and Morales gave a fully geometric proof for the volume formula and a part generating function proof for the Ehrhart polynomial formula. The goal of the present paper is to provide a fully geometric proof for the Ehrhart polynomial formula of flow polytopes.
\end{abstract}

\section{Introduction}

Polytopes are  ubiquitous in mathematics.  Two
immediate questions about any  integer polytope $P$ are to compute its  volume and  the number of integer points  in $P$ and its dilations. The Baldoni--Vergne  formulas (Theorem \ref{thmA}) answer these questions for flow polytopes. This paper is concerned with understanding the aforementioned formulas geometrically, as their original proof \cite{bv} is via residues and only a partial geometric proof is known to date \cite{MMlidskii}.
 
 Flow polytopes are   fundamental in combinatorial optimization  \cite{schrijver, counting}.   Postnikov and Stanley discovered the  connection of volumes of flow polytopes  to Kostant partition functions (unpublished; see \cite{bv, mm}), inspiring the work of    Baldoni and Vergne  \cite{bv}. 
 Flow polytopes are also related to Schubert and  Grothendieck polynomials \cite{groth} and  the space of diagonal harmonics \cite{tesler, diag}.
 
 The connection between flow polytopes and Kostant partition functions is a motivating force of this paper. While it is abundantly clear from the definition of a flow polytope (given in \eqref{eq:deff} below) that the number of its integer points is an enumeration of the Kostant partition function, the relation of its volume to the Kostant partition function is less than obvious. Before we explain the above, we define the Kostant partition function and highlight its importance.
 
 The \textbf{Kostant partition function} $K_n({\bf a})$ of type $A_n$ is  the number of ways to write  the  vector  ${\bf a}=(a_1, \ldots, a_{n+1}) \in \Z^{n+1}$ as a nonnegative integral combination of the vectors $e_i-e_j$ for   $1\leq i<j\leq n+1$,  where $e_i$ is the $i$-th standard basis vector in $\R^{n+1}$. It is a special vector partition function introduced by Bertram Kostant in 1958 in order to get an expression for the multiplicity of a weight of an irreducible representation of a semisimple Lie algebra \cite{kostant2, kostant1}, now known as the Weyl character formula. Kostant partition functions   appear not only in representation theory, but in algebraic combinatorics, toric geometry and approximation theory, among other areas.

 The Kostant partition function is a piecewise polynomial function \cite{sturm}, whose domains of polynomiality are maximal convex cones in the common refinement of all triangulations of the convex hull of the positive roots \cite{deloerasturm}. Despite the above description of the  domains of polynomiality of the Kostant partition function, enumerating these domains has remained elusive \cite{ubiquity, deloerasturm, GMP}. In this paper we will be concerned with the Kostant partition function and its generalizations evaluated at vectors ${\bf a}=(a_1, \ldots, a_{n+1}) \in \Z^{n+1}$, where $a_i \geq 0$ for all $i \in [n]$. These vectors form the ``nice chamber" \cite{bv}, which is a distinguished  domain of polynomiality of the Kostant partition function.

  We now define our main geometric object, the flow polytope, as well as relate it to the generalized Kostant partition function $K_G(\cdot)$ defined below.

Let $M_G$ denote the incidence matrix of the graph $G$  on the vertex set $[n+1]$; that is let the columns of $M_G$ be the vectors $e_i-e_j$ for $(i,j)\in E(G)$, $i<j$,  where $e_i$ is the $i$-th standard basis vector in $\R^{n+1}$. Then, the {\bf flow polytope} $\mathcal{F}_G({\bf a})$ associated to the graph $G$ and the {\bf netflow vector} ${\bf a}=(a_1, \ldots, a_{n+1}) \in \Z^{n+1}$ is defined as  
	\begin{equation} \label{eq:deff}\mathcal{F}_G({\bf a})= \{f\in \R^{|E(G)|}_{\geq 0}:\, M_Gf={\bf a} \}. \end{equation}

The flows $f$ given in \eqref{eq:deff} are also referred to as $\bf{a}$-flows on the graph $G$. 

 The   {\bf normalized volume} of a
$d$-dimensional polytope $P \subset \mathbb{R}^n$ is the volume form $\vol(\cdot)$ that assigns a volume of one to the smallest $d$-dimensional simplex whose vertices are in the lattice equal to the intersection of $\Z^n$ with the affine span of the polytope $P$.  
  The number of lattice points of  the {$t^{th}$ dilate} of ${P} \subset \R^n$, $t {P}:=\{(tx_1, \ldots, tx_{n}) \mid  (x_1, \ldots, x_{n}) \in {P}\}$, is given by the {\bf Ehrhart function} ${\rm Ehr}({P}, t)$.
 If ${P}$ has  integral vertices then ${\rm Ehr}({P}, t)$ is a  polynomial. The leading coefficient of the Ehrhart polynomial ${\rm Ehr}({P}, t)$ is $\dim(P)! \vol(P)$.
  
 Note that the number of integer points in $\mathcal{F}_G({\bf a})$ is exactly the number of ways to write ${\bf a}$ as a nonnegative integral combination of the vectors $e_i-e_j$ for edges $(i,j)$ in $G$, $i<j$, that is the \textbf{generalized Kostant partition function} $K_G({\bf a})$. 	It thus follows that  ${\rm Ehr}(\F_{G}({\bf a}),t)=K_G(t {\bf a})$.  The classical Kostant  partition function   $K_n({\bf a})$ corresponds to the case of the complete graph $K_{n+1}$. Following Baldoni and Vergne \cite{bv } for brevity we will simply refer to the generalized Kostant partition function as the Kostant partition function.

  The magic of the  Baldoni--Vergne formulas is that  for flow polytopes $\F_{G}({\bf a})$, their Ehrhart polynomial \newline ${\rm Ehr}(\F_{G}({\bf a}),t)=K_G(t {\bf a})$ can be deduced from their volume function!

\begin{theorem}[Baldoni--Vergne  formulas {\cite[Thm. 38]{bv}}]
\label{thmA} 
Let $G$ be a connected graph on the vertex set  $[n+1]$, with $m = |E(G)|$ edges directed $i\to j$ when $i < j$, with at least one outgoing edge at
vertex $i$ for $i=1,\ldots,n$, and let
$\aa=(a_1,\ldots,a_n,-\sum_{i=1}^n a_i)$, $a_i \in \mathbb{Z}_{\geq 0}$. Then
\begin{align} \label{eq:vol}
\vol \F_G({\bf a}) &= \sum_{{\bf j}}
\binom{m-n}{j_1,\ldots,j_n} a_1^{j_1}\cdots
a_n^{j_n}\cdot K_{G}\left(j_1-\lout_G(1), \ldots, j_n - \lout_G(n),0\right),     \\
K_{G}({\bf a}) &= \sum_{{{\bf j}}}
\multiset{a_1-\lin_G(1)}{j_1}\cdots
\multiset{a_{n}-\lin_G(n)}{j_{n}} \cdot  K_{G}\left(j_1-\lout_G(1), \ldots, j_n - \lout_G(n),0\right), \label{eq:kost}    
\end{align} 
for $\lout_G(i)=\outdeg_G(i)-1$ and $\lin_G(i)=\indeg_G(i)-1$ where $\outdeg_G(i)$ and $\indeg_G(i)$ denote
the outdegree and indegree of  vertex $i$ in $G$. Each sum is over weak
compositions ${\bf j}=(j_1,j_2,\ldots,j_n)$ of $
m-n$ that are $\geq (\lout_G(1),\ldots,\lout_G(n))$ in dominance order (that is $\sum_{k=1}^l j_k \geq \sum_{k=1}^l \lout_G(k)$ for all $l \in [n]$) and $\multiset{n}{k}:=\binom{n+k-1}{k}$.
\end{theorem}

 The proof provided by Baldoni--Vergne \cite{bv} for Theorem \ref{thmA} relies on residue computations, leaving the combinatorial nature of their formulas a mystery. The aim of the authors in  \cite{MMlidskii}   was to demystify Theorem \ref{thmA}  by proving it via polytopal subdivisions of $\F_G({\bf a})$.    They do this  by  constructing  a special subdivision of $\F_G({\bf a})$ referred to as the {\it canonical subdivision}, which allows for a geometric computation of  the volume of $\F_G({\bf a})$. In order to deduce   \eqref{eq:kost} the   generating functions of the Kostant partition functions   are also used in  \cite{MMlidskii}.  While the use of the aforementioned generating functions in \cite{MMlidskii}  is natural, our goal and result in the present paper is to avoid them and    give a purely geometric proof of  \eqref{eq:kost}.
 
  \bigskip
  
\noindent \textbf{Outline of the paper.} Section \ref{sec:def} explains subdivisions of flow polytopes,  Section \ref{sec:?} provides further polytopal insights and  Section \ref{sec:geom} provides a new, completely geometric proof of   \eqref{eq:kost}. We conclude in Section \ref{sec:con} with general remarks.

\section{Subdividing flow polytopes}
\label{sec:def}
   	
The guiding principle beneath subdivisions of polytopes is a simple one: we aim to subdivide polytopes into smaller ones in hopes of using our understanding of the smaller polytopes to gain understanding of the polytope we started with. For example, we may be interested in the volume of a polytope $P$, and one way to calculate it would be if we could count the top dimensional simplices of a unimodular triangulation of $P$ (provided one exists). This is exactly what Morales and the second author of this paper accomplish for flow polytopes in order to prove their volume formula \eqref{eq:vol} geometrically in \cite{MMlidskii}. However,  understanding the top dimensional  	simplices of a unimodular triangulation of $P$ is not sufficient for counting the number of integer points of $P$, since we cannot simply sum over the top dimensional simplices as we do for volume! This is why  getting a geometric proof of \eqref{eq:kost} requires further insights. The main insight is the realization that we can reinterpret the left hand side of  \eqref{eq:kost} as a volume of a flow polytope (different from $\F_G(\g)$), and then the right hand side can be obtained by summing volumes of polytopes in  a subdivision of our new flow polytope.  This way we do not have the issue of overcounting integer points on the intersections of the polytopes in a subdivision of  $\F_G(\g)$! More details on this construction are coming  in Section \ref{sec:?}. This section is devoted to reviewing and generalizing the subdivision construction of \cite{MMlidskii}, whose exposition we follow. 

\medskip

The crucial lemma  that we are building up to in this section  is the \textit{Subdivision Lemma}, Lemma \ref{lem:sub}. The following sequence of definitions are necessary in order to understand the right hand side of  \eqref{eq:?}.

\medskip

A {\bf bipartite noncrossing tree}  is a  tree with a
distinguished bipartition of vertices into {\bf left vertices}
$x_1,\ldots,x_{\ell}$ and {\bf right vertices $x_{\ell+1},\ldots,
  x_{\ell+r}$} with no pair of edges $(x_p,x_{\ell+q}),
(x_t,x_{\ell+u})$ where $p<t$ and $q>u$. 
Denote by
$\mathcal{T}_{L,R}$ the set of  bipartite noncrossing trees
where $L$ and $R$ are the ordered
sets $(x_1,\ldots,x_{\ell})$ and $(x_{\ell+1},\ldots,x_{\ell+r})$ respectively. Note that $\#
\mathcal{T}_{L,R}=\binom{\ell+r-2}{\ell-1}$.

\begin{figure} 
\begin{center}\includegraphics[scale=.75]{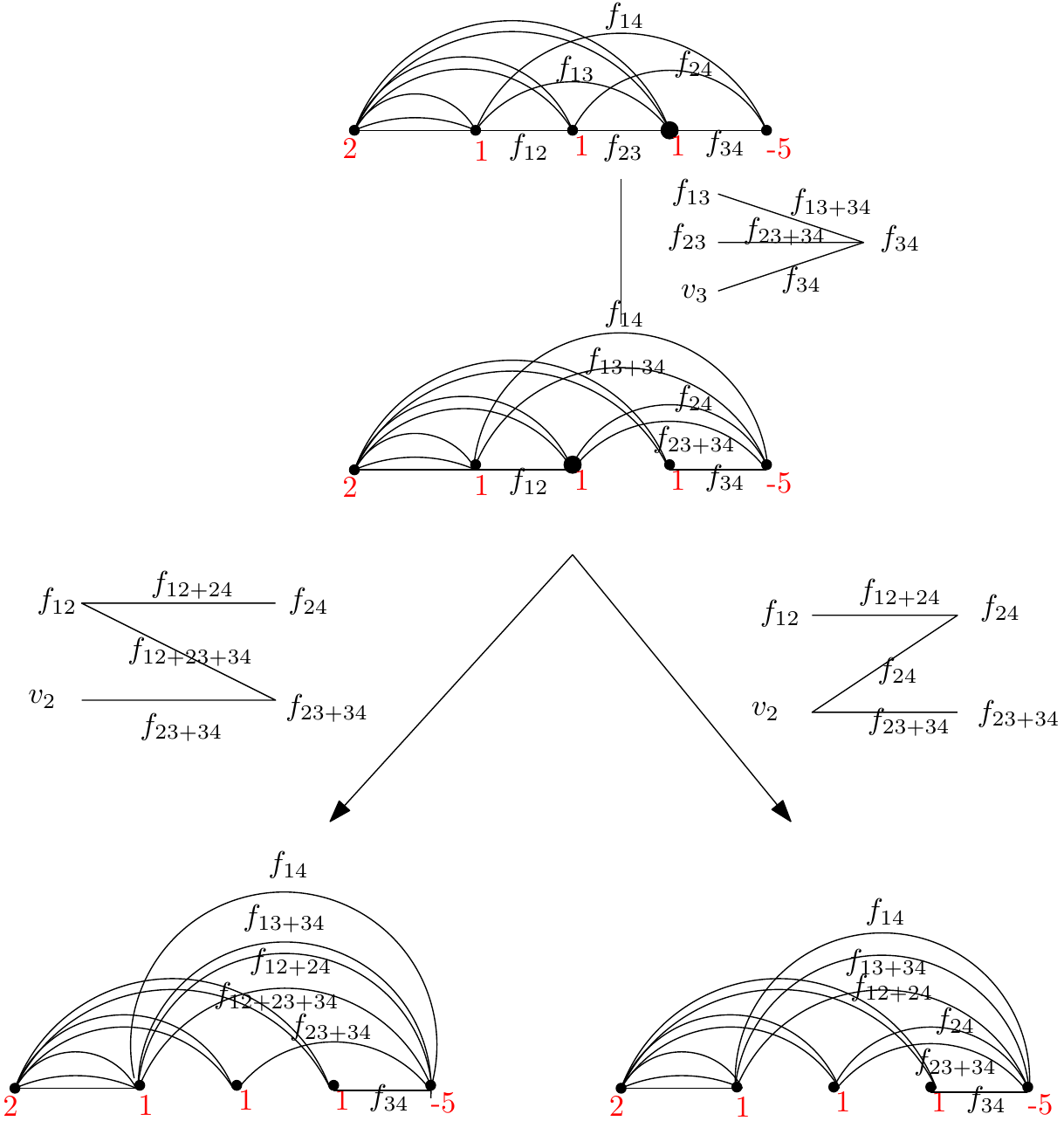}\end{center}
\caption{Reduction tree with root labeled by the graph $K_4(3,2,2):=([0,4], E(K_4)\cup \{\{(0,1), (0,1), (0,1), (0,2), (0,2), (0,3), (0,3)\}\})$. The notation $f_{ij+jk}$ stands for the formal sum of edges $f_{ij}+f_{jk}$.  The vertex
  of the graph
  where the     reduction is taking place is enlarged. The flow polytopes corresponding to the leaves of the     reduction tree  dissect the flow polytope corresponding to the root of the tree, see Lemma~\ref{lem:sub}.}
\label{fig:complete}
\end{figure}

 Consider a graph $G$ on the vertex set $[n+1]$ with edges oriented from smaller to larger vertices and an integer netflow
vector $\g=(a_1, \ldots, a_n, -\sum_i a_i)$, with  $a_i\geq 0$, $i \in [n]$.  Pick an arbitrary vertex
$i$,$1<i<n+1$, of $G$ as well as a submultiset  $\mathcal{I}_i$ of the  multiset of {incoming edges} to $i$ and submultiset $\mathcal{O}_i$ of the multiset of {outgoing edges} from $i$.  Given an ordering on the sets $\mathcal I_i$ and $\mathcal O_i$ and a bipartite noncrossing tree $T \in \mathcal T_{\mathcal I_i\cup\{i\},\mathcal O_i}$, where $\mathcal I_i\cup\{i\}$ is ordered according to the order on $\mathcal I_i$ with $i$ appended as its last element,we describe the construction of new graphs $G^{(i)}_T(\mathcal{I}_i, \mathcal{O}_i)$ from $G$ as follows.

For each tree-edge $(e_1,e_2)$ of $T$ where $e_1=(r,i) \in \mathcal{I}_i$ and $e_2=(i,s)\in \mathcal{O}_i$ let $edge(e_1,e_2)=(r,s)$ and we let  $edge(i,
(i,j))=(i,j)$. We think of $edge(e_1,e_2)$ as a formal {\bf sum of the edges} $e_1$ and $e_2$, where $edge(i,
(i,j))=(i,j)$ as the edge $(i,j)$.

The graph $G^{(i)}_T(\mathcal{I}_i, \mathcal{O}_i)$ is then defined as the graph obtained from $G$
by deleting all edges in $\I_i \cup \O_i$ of $G$  and adding  the
multiset of edges $\{\{edge(e_1,e_2) ~|~ (e_1,e_2)\in E(T)\}\} \cup \{\{edge(i,(i,j)) ~|~ (i,(i,j))\in E(T)\}\}.$ 
 See Figure~\ref{fig:complete}.

\medskip

The difference in the above and that of \cite[Section 3]{MMlidskii} is that in \cite{MMlidskii} the multisets $\mathcal{I}_i$ and $\mathcal{O}_i$ are always taken to equal to the multiset of incoming and the multiset of outgoing edges of $i$, whereas here we allow them to be proper submultisets of the multiset of incoming and the multiset of outgoing edges of $i$. Note that in the example on  Figure \ref{fig:complete} we have that 
$\mathcal{I}_3$ and $\mathcal{I}_2$ are proper subsets of the incoming edges at vertices 3, and 2, respectively. 

  The \textit{Subdivision Lemma}, Lemma \ref{lem:sub}, states that $\F_G(\g)$ is a union  over $T \in \mathcal{T}_{\mathcal{I}_i \cup \{i\},\mathcal{O}_i}$ of the smaller polytopes  $\phi_T(\F_{G^{(i)}_T(\mathcal{I}_i, \mathcal{O}_i)}(\g)),$ where $\phi_T$ is an  integral equivalence between $\F_{G^{(i)}_T(\mathcal{I}_i, \mathcal{O}_i)}(\g)$ and its image $\phi_T(\F_{G^{(i)}_T(\mathcal{I}_i, \mathcal{O}_i)}(\g))$. We now define integrally equivalence of polytopes and   the maps $\phi_T$. 
  
Two polytopes $P_1\subseteq \R^{k_1}$ and ${P_2\subseteq \R^{k_2}}$ are \textbf{integrally equivalent} if there is an affine transformation $t:\R^{k_1}\to \R^{k_2}$ that is a bijection $P_1\to P_2$ and a bijection $\aff(P_1)\cap \mathbb{Z}^{k_1}\to \aff(P_2)\cap \mathbb{Z}^{k_2}$. Integrally equivalent polytopes have the same face lattice, volume, and Ehrhart polynomial. 

 Fix $T \in \mathcal{T}_{\mathcal{I}_i \cup \{i\},\mathcal{O}_i}$.  Recall that  $\F_{G^{(i)}_T(\mathcal{I}_i, \mathcal{O}_i)}(\g) \subset \R^{|E(G^{(i)}_T(\mathcal{I}_i, \mathcal{O}_i))|}$ and denote  the coordinates of  $\R^{|E(G^{(i)}_T(\mathcal{I}_i, \mathcal{O}_i))|}$ by $({\rm coord}_e)_{e \in E(G^{(i)}_T(\mathcal{I}_i, \mathcal{O}_i))}$; moreover,  $\F_G(\g)\subset \R^{|E(G)|}$ and denote  the coordinates of  $\R^{|E(G)|}$ by $({\rm coord}_d)_{d \in E(G)}$.  Recall that each edge of $e \in G^{(i)}_T(\mathcal{I}_i, \mathcal{O}_i)$ is a sum of (one or more) edges of  the
original graph $G$; denote by $s(e)$ the subset of edges of $G$ which we sum in order to get $e$. 
Define the affine transformation $\phi_T:\R^{|E(G^{(i)}_T(\mathcal{I}_i, \mathcal{O}_i))|}\to \R^{|E(G)|}$ via   
\[
\phi_T(({\rm c}_e)_{e \in E(G^{(i)}_T(\mathcal{I}_i, \mathcal{O}_i))})=({\rm c}_d)_{d \in E(G)}, \qquad \text{where} \qquad {\rm c}_d=\sum_{e: d \in s(e) }{\rm c}_e.
\]
Note that $\phi_T$ is an invertible linear map between vector spaces of the same dimension which restricts to a bijection on the underlying lattice. Therefore $\phi_T$ is an integral equivalence between $\mathcal F_{G_T^{(i)}}(\mathcal I_i, \mathcal O_i)$ and its image in $\R^{|E(G)|}$. By definition,  $\phi_T(\F_{G^{(i)}_T(\mathcal{I}_i, \mathcal{O}_i)}(\a))\subseteq \F_{G}(\a)$. An illustration of the map $\phi_T$ appears on Figure~\ref{fig:phiT}.

\begin{figure} 
\begin{center}\includegraphics[scale=.75]{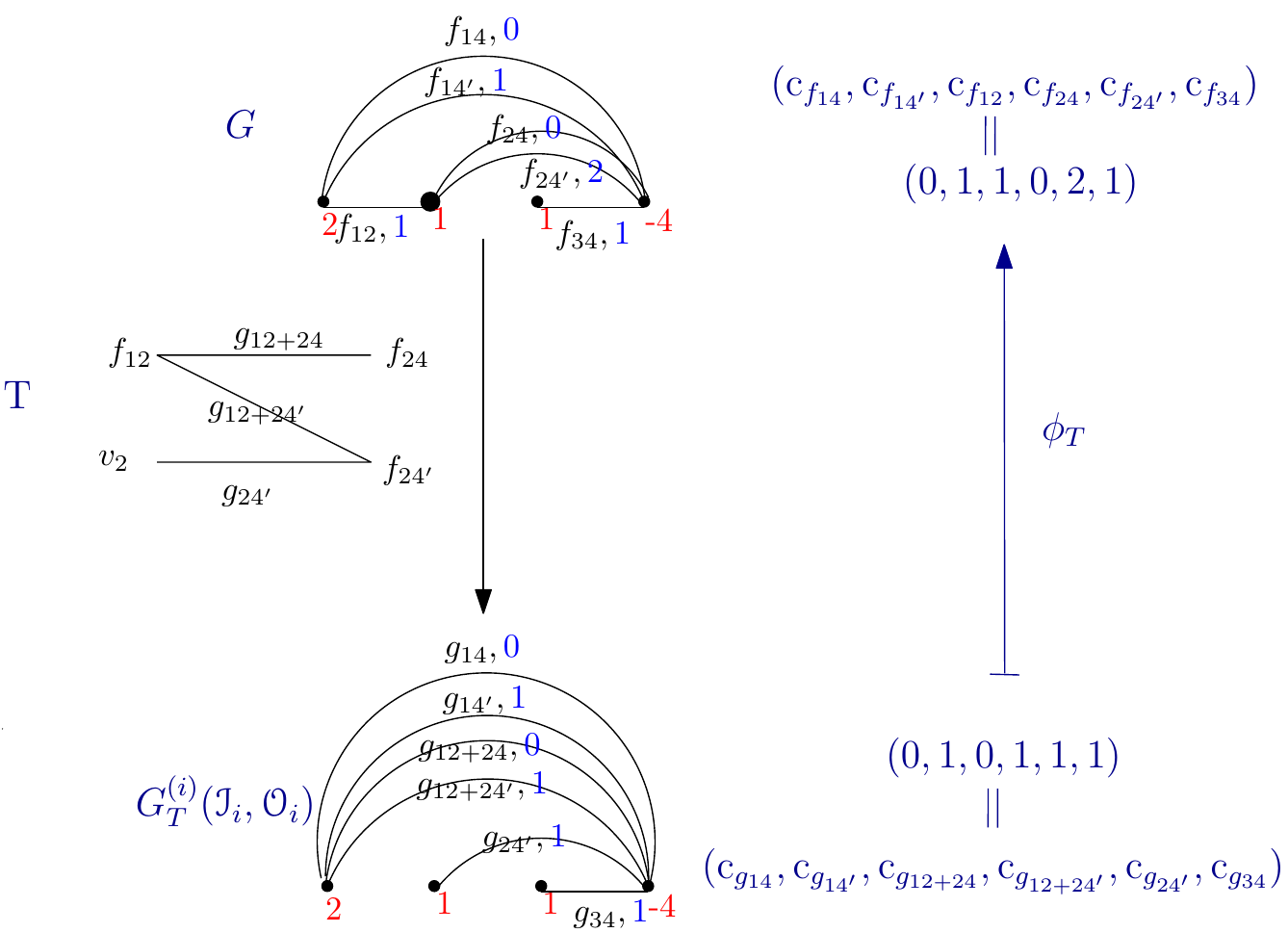}\end{center}
\caption{Illustration of  $\phi_T (0,1,0,1,1,1)=(0,1,1,0,2,1)$. Note that we ordered the coordinates of $\R^{|E(G^{(i)}_T(\mathcal{I}_i, \mathcal{O}_i))|}$ (where $i=2$, $\mathcal{I}_i=(f_{1,2})$ and $\mathcal{O}_i)=(f_{24}, f_{24'})$) by ordering the edges of $G^{(i)}_T(\mathcal{I}_i, \mathcal{O}_i)$ in the order   $({\rm c}_{g_{14}}, {\rm c}_{g_{14'}}, {\rm c}_{g_{12+24}}, {\rm c}_{g_{12+24'}}, {\rm c}_{g_{24'}}, {\rm c}_{g_{34}})$; we ordered the coordinates of $\R^{|E(G)|}$ by ordering the edges of $G$  in the order $({\rm c}_{f_{14}}, {\rm c}_{f_{14'}}, {\rm c}_{f_{12}}, {\rm c}_{f_{24}}, {\rm c}_{f_{24'}}, {\rm c}_{f_{34}})$. We have that $s({g_{14}})=\{f_{14}\}, s({g_{14'}})=\{f_{14'}\}, s({g_{12+24}})=\{f_{12}, f_{24}\},$ $s({g_{12+24'}})=\{f_{12}, f_{24'}\}, s({g_{24'}})=\{f_{24'}\}, s({g_{34}})=\{f_{34}\}$.}
\label{fig:phiT}
\end{figure}

   By abuse of notation instead of writing $\phi_T(\F_{G^{(i)}_T(\mathcal{I}_i, \mathcal{O}_i)}(\a))$ we write  $\F_{G^{(i)}_T(\mathcal{I}_i, \mathcal{O}_i)}(\a)$  from now on, including in Lemma \ref{lem:sub}. With this convention we have 
$\F_{G^{(i)}_T(\mathcal{I}_i, \mathcal{O}_i)}(\a)\subseteq \F_{G}(\a)$. 
\medskip

The following Subdivision Lemma generalizes \cite[Lemma 3.4]{MMlidskii}. The proof is analogous to that of \cite[Lemma 3.4]{MMlidskii}, and we leave it to the interested reader.

\begin{lemma}[Subdivision Lemma]  \label{lem:sub}   Let  $G$ be a graph on the vertex set $[n+1]$.
 Fix an integer netflow vector $\aa=(a_1,\ldots,a_n,-\sum_{i=1}^n a_i)$, $a_i \in \mathbb{Z}_{\geq 0}$ as well as a vertex $i\in \{2,\ldots,n\}$ and ordered multisets $\mathcal{I}_i, \mathcal{O}_i$, which are submultisets of the multiset of incoming and outgoing edges incident to $i$. Then, 
\begin{equation} \label{eq:?}
\F_G(\g)=\bigcup_{T \in \mathcal{T}_{\mathcal{I}_i \cup \{i\},\mathcal{O}_i}} \F_{G^{(i)}_T(\mathcal{I}_i, \mathcal{O}_i)}(\g).
\end{equation} 
Moreover,  $\{\F_{G^{(i)}_T(\mathcal{I}_i, \mathcal{O}_i)}(\g)\}_{T \in \mathcal{T}_{\mathcal{I}_i \cup \{i\}, \mathcal{O}_i}}$ are interior disjoint.  
\end{lemma}

We refer to replacing $G$ by $\{G^{(i)}_T(\mathcal{I}_i, \mathcal{O}_i)\}_{T \in \mathcal{T}_{\mathcal{I}_i\cup \{i\}, \mathcal{O}_i}}$
as in Lemma \ref{lem:sub}  as a \textbf{reduction}. We can encode
a series of     reductions on a flow polytope $\F_{G}(\a)$ in a rooted
tree called a \textbf{reduction tree} with root $G$; see Figure \ref{fig:complete} for an example. The root of this tree is the original graph
$G$.  After doing reductions on vertex $i$ with fixed $\mathcal{I}_i, \mathcal{O}_i$  ordered submultisets of the multiset of incoming and outgoing edges incident to $i$, the descendant nodes of the root are the graphs
${G^{(i)}_T(\mathcal{I}_i, \mathcal{O}_i)}$, for $T \in \mathcal{T}_{\mathcal{I}_i \cup \{i\},\mathcal{O}_i}$. For each new node we decide whether to stop or repeat this
process to define its descendants. The leaves of the reduction tree are those with no children.   Note that the flow polytopes
$\F_H(\a)$ of the graphs $H$ at the leaves of the     reduction tree   are interior disjoint and their union is $\F_{G}(\a)$   by  repeated application of  Lemma \ref{lem:sub}. 

 \medskip
 
 In \cite{MMlidskii} the authors used their less general version of Lemma \ref{lem:sub} to define the {\bf canonical subdivision} of flow polytopes $\mathcal F_{G}({\bf a})$.  This allowed them in particular to derive \eqref{eq:vol} purely geometrically. We include their construction here and will use it in the next section.

\begin{definition} \label{def:rg1} The {\bf canonical reduction tree} $R_G$ for a graph $G$ on the vertex set $[n+1]$ is obtained by   repeated use of Lemma \ref{lem:sub}   on the vertices $n, n-1, \ldots, 2$ in this order and on the sets of edges $\mathcal{I}_i=\{\{(j,i) \in E(G) \mid j<i\}\}$ and  $\mathcal{O}_i=\{\{(i,j) \in E(G) \mid i<j\}\}$, $i \in \{n, n-1, \ldots, 2\},$ where both $\mathcal{I}_i$ and $\mathcal{O}_i$ are ordered by decreasing edge lengths. \end{definition}

	For an example of a canonical reduction tree see Figure~\ref{fig:canonical}. Note that at each vertex $i$ the set $\mathcal{I}_i$ is all of the coming edges at $i$ (unlike in Figure \ref{fig:complete}) and the set $\mathcal{O}_i$ is the set of all outgoing edges at $i$. 

\begin{figure} 
\begin{center}\includegraphics[scale=.70]{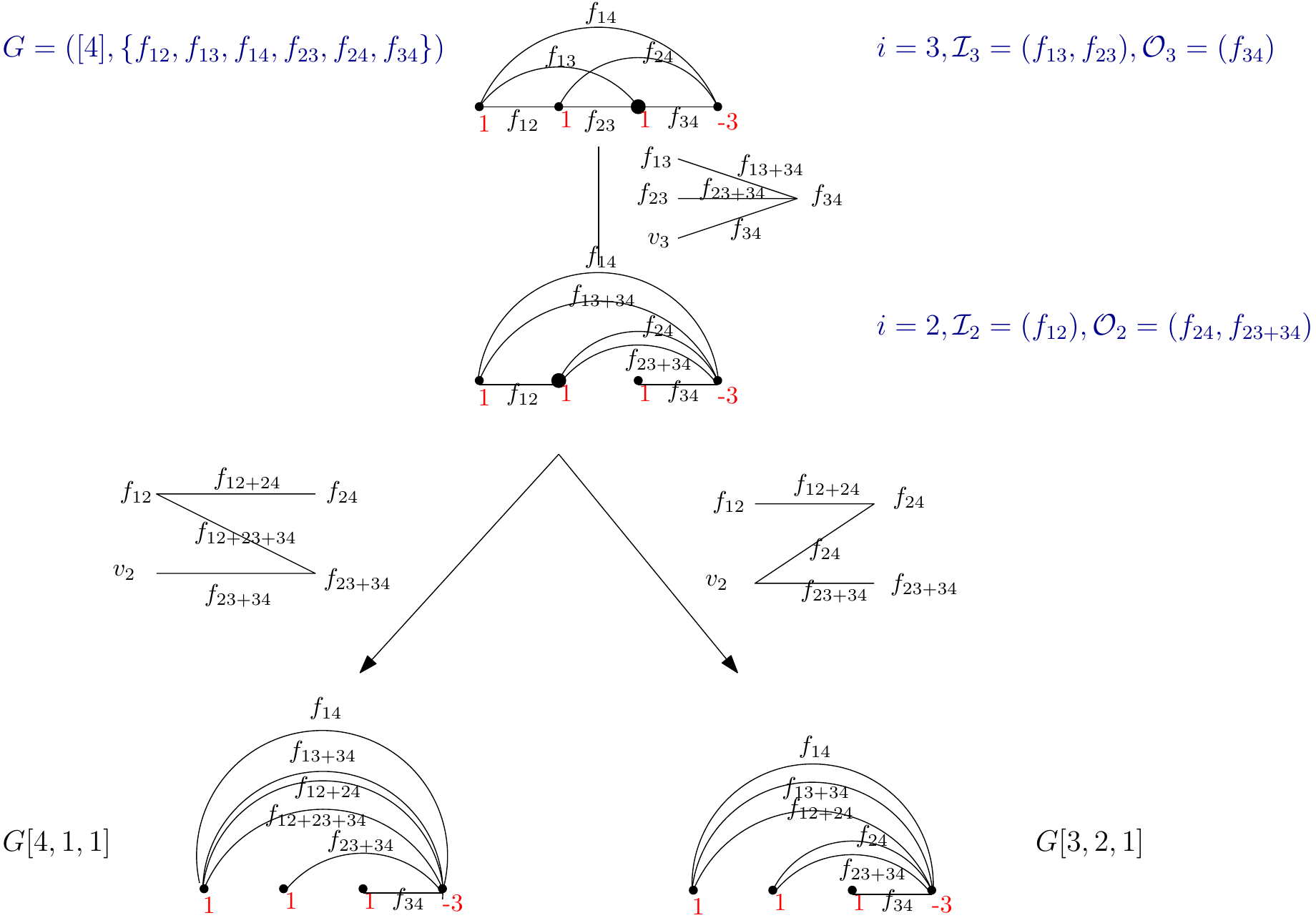}\end{center}
\caption{The canonical reduction tree for complete graph $K_4$ with nonnegative netflows on the first 3 vertices.}
\label{fig:canonical}
\end{figure}

\begin{definition} \label{def:gm} 	   Given a tuple $\m=(m_1,\ldots,m_n)$ of positive
integers, let $G[\m]$ be the graph with vertices $[n+1]$ and $m_i$
edges $(i,n+1)$. 
\end{definition}
	
	Note that the leaves of the canonical reduction tree in Figure \ref{fig:canonical} are both of the form $G[\m]$ for some $\m$; in particular the leaf on the left is $G[4,1,1]$ and the leaf on the right if $G[3,2,1]$. The following theorem states that this is no coincidence. 
	
	Recall that given graph $G$ on the vertex set $[n+1]$ we let  $\lout_G(i)=\outdeg_G(i)-1$,  where $\outdeg_G(i)$   is 
the outdegree   of  vertex $i$ in  $G$. We denote  $\lout_G=( \lout_G(1), \ldots,  \lout_G(n))$.
	 
\begin{theorem} \label{thm:a} \cite[Section 4]{MMlidskii}  
The  canonical reduction tree $R_G$ of  $G$ on the vertex set $[n+1]$ with $m$ edges has  $$\sum_{\substack{{\bf j} \geq \lout_G \\ j_1 + \dots + j_n = m-n}} K_G(j_1 - \lout_G(1), \dots, j_n - \lout_G(n), 0)$$  leaves, where:
\begin{itemize}

\item the sum is over  weak
compositions ${\bf j}=(j_1,j_2,\ldots,j_n)$ of $
m-n$ that are $\geq (\om(1),\ldots,\om(n))$ in dominance order, that is, $\sum_{k=1}^l j_k \geq \sum_{k=1}^l \lout_G(k)$ for all $l \in [n]$; 
\item  $K_G(j_1 -  \lout_G(1), \dots, j_n -  \lout_G(n), 0)$ of the leaves of $R_G$  are   ${G[{\bf j+1}]}$;
\item  the polytopes $\mathcal F_{G[{\bf j+1}]}({\bf a})$     are interior disjoint and their union is $\mathcal F_{G}({\bf a})$;

\item  if $a_i>0$, $i \in [n]$, then  the polytopes $\mathcal F_{G[{\bf j+1}]}({\bf a})$    are of the same dimension as $\mathcal F_{G}({\bf a})$.
\end{itemize}
\end{theorem}

The polytopes specified in Theorem \ref{thm:a} are the top dimensional polytopes in the {\bf canonical subdivision} of $\mathcal F_{G}({\bf a})$ \cite{MMlidskii}.  

\begin{example} \label{ex:1} Let us apply Theorem \ref{thm:a} to the canonical reduction tree $R_{K_4}$ from Figure~\ref{fig:canonical}. Here $n=3$ and $m=6$. 
We have that $(\outdeg_{K_4}(1), \outdeg_{K_4}(2), \outdeg_{K_4}(3))=(3,2,1)$ and thus, $(\lout_{K_4}(1), \lout_{K_4}(2), \lout_{K_4}(3))=(2,1,0)$. The sum in Theorem  \ref{thm:a}  is over  compositions ${\bf j}=(j_1,j_2,j_3)$ of $
3$ that are $\geq (2,1,0)$ in dominance order. Thus, the two possible ${\bf j}$'s are $(3,0,0)$ and $(2,1,0)$. The first two points of  Theorem  \ref{thm:a} thus state that there are $K_{K_4}(3 - 2, 0-1, 0-0, 0)=1$   leaves of $R_{K_4}$ that  are   ${G[4,1,1]}$ and  $K_{K_4}(2 - 2, 1-1, 0-0, 0)=1$   leaves of $R_{K_4}$ that  are   ${G[3,2,1]}$; see Figure~\ref{fig:canonical}. 

\end{example}

\section{A few geometric insights} 
\label{sec:?}
This section collects the main insights necessary for proving \eqref{eq:kost} purely geometrically. The proof of  \eqref{eq:kost} relies on stringing all the following statements together in order to give a proof of it in Theorem \ref{thm:main-sec-4}.  As mentioned in the previous section: we  reinterpret the left hand side of  \eqref{eq:kost} as a volume of a flow polytope (namely, of $\mathcal F_{\Gc}(e_1 - e_{n+2})$, see Definition \ref{def:c} for the meaning of $\Gc$), and then the right hand side can be obtained by summing volumes of polytopes in  a subdivision of our new flow polytope.  

\begin{definition} \label{def:c}
Fix a vector ${\bf c} = (c_1, \dots, c_n)\in \Z^n_{>0}$  and a graph $G$ on the vertex set $[n+1]$. The graph $\Gc$ is defined to be the graph obtained by adding a source vertex $0$ to $V(G)$, so that $V(\Gc) = [0,n+1]$, along with $c_i$ edges edges $(0, i)$, for every $i \in [n]$, to $E(G)$. Formally, we have
$$G({\bf c}) := (V(G)\cup \{0\}, E(G) \cup \{\{(0, i)^{c_i} \mid  i \in [n]\}\}),$$
where $(0, i)^{c_i}$ signifies $c_i$ copies of the edge $(0, i)$. Note  that the graph $\Gc$ restricted to the vertex set $[n+1]$ is equal to the graph $G$.\end{definition}

Of importance in Definition \ref{def:c} is that when we restrict $\Gc$ to the vertex set $[n+1],$ we get $G$ back. 
This allows us to use our full knowledge of canonical reduction trees, as given in Theorem \ref{thm:a}.  This is because Lemma \ref{lem:sub} allows for not using all the incoming (or outgoing) edges at a vertex -- and yet if we use all incoming and outgoing edges that are in a given subgraph of the graph,  we can still invoke Theorem \ref{thm:a} for the mentioned subgraph! Details are spelled out in Theorem \ref{thm:dis} below.
\medskip 

Recall that given graph $G$ on the vertex set $[n+1]$ we let    $\lin_G(i)=\indeg_G(i)-1$, where  $\indeg_G(i)$ is
the   indegree of  vertex $i$ in  $G$.

With Definition \ref{def:c} we have that:
\begin{lemma}
\label{lem:in-vector}
Fix a vector ${\bf c} = (c_1, \dots, c_n)\in \Z^n_{>0}$ and a graph $G$ on the vertex set $[n+1]$.  Define $a_i\colonequals \inn_G(i) + c_i$. The number of   integer points in  $\mathcal F_G(a_1, \dots, a_n, -\sum_{i=1}^na_i)$ is equal to the number of integer points in the flow polytope $\mathcal F_{G({\bf c})}(0, \inn_{\Gc}(1), \dots, \inn_{\Gc}(n), -\sum_{i=1}^n \inn_{\Gc}(i))$. In other words,
$$K_G\B(a_1, \dots, a_n, -\sum_{i=1}^na_i\B) = K_{\Gc}\B(0, \inn_{\Gc}(1), \dots, \inn_{\Gc}(n), -\sum_{i=1}^n \inn_{\Gc}(i)\B).$$
\end{lemma}

\begin{proof}
Consider an integer $(0,  \inn_{\Gc}(1), \dots,  \inn_{\Gc}(n), -\sum_{i=1}^n  \inn_{\Gc}(i))$-flow $f$  on the graph $G({\bf c})$.  Note that when we restrict $f$ to $G\subset \Gc$  it is  an $( \inn_{\Gc}(1), \dots,  \inn_{\Gc}(n), -\sum_{i=1}^n  \inn_{\Gc}(i))$-flow on $G$. By the definition of $\Gc$, we have $ \inn_{\Gc}(i) = \inn_{G}(i) + c_i = a_i$. Thus, an integer $(0,  \inn_{\Gc}(1), \dots,  \inn_{\Gc}(n), -\sum_{i=1}^n  \inn_{\Gc}(i))$-flow on ${G({\bf c})}$   restricts to an integer $(a_1, \dots, a_n,$ $-\sum_{i=1}^na_i)$-flow on $G$. This is clearly a bijection showing that the number of integer $(0,  \inn_{\Gc}(1), \dots,$ $ \inn_{\Gc}(n), -\sum_{i=1}^n  \inn_{\Gc}(i))$-flows on ${G({\bf c})}$   equals to the number of  integer $(a_1, \dots, a_n, -\sum_{i=1}^na_i)$-flows on $G$. The latter equal $K_{\Gc}(0,  \inn_{\Gc}(1), \dots,  \inn_{\Gc}(n), -\sum_{i=1}^n  \inn_{\Gc}(i))$ and $K_G(a_1, \dots, a_n, -\sum_{i=1}^na_i)$, respectively.  
\end{proof}

\subsection{Dissecting  $\mathcal F_{\Gc}(e_1 - e_{n+2})$.} In this section we show how to dissect $\mathcal F_{\Gc}(e_1 - e_{n+2})$, ${\bf c}\in \Z^{n}_{>0}$,    into $$\displaystyle \sum_{\substack{{\bf j} \geq \lout_G \\ j_1 + \dots + j_n = m-n}} \frac{(c_1)_{j_1}}{j_1!}\dots\frac{(c_n)_{j_n}}{j_n!} K_G(j_1 -  \lout_{G}(1), \dots, j_n -  \lout_{G}(n), 0)$$ many unimodular simplices. The notation $(k)_j$ stands for $(k)_j:=k(k-1)\cdots (k-j+1)$.

	\medskip

\begin{definition} \label{def:rg2} Given a graph $G$ on the vertex set $[n+1]$   define the reduction tree $R_G^{\bf c}$ with root  $\Gc$ as the reduction tree obtained by   repeated use of Lemma \ref{lem:sub}   on the vertices $n, n-1, \ldots, 2$ in this order and on the sets of edges $\mathcal{I}_i=\{\{(j,i) \in E(G) \mid j<i\}\}$ and  $\mathcal{O}_i=\{\{(i,j) \in E(G) \mid i<j\}\}$, $i \in \{n, n-1, \ldots, 2\},$ where both $\mathcal{I}_i$ and $\mathcal{O}_i$ are ordered by decreasing edge lengths. 
\end{definition}

We note that Definition \ref{def:rg2} is set up so that if we delete all edges incident to $0$ in the graphs labeling the nodes of $R_G^{\bf c}$ we obtain the canonical reduction tree $R_G$ of $G$ as  in Definition \ref{def:rg1}. For an example of the reduction tree $R_{K_4}^{(3,2,2)}$ see Figure \ref{fig:complete}; compare this with the canonical reduction tree $R_{K_4}$ on Figure \ref{fig:canonical}.

\begin{theorem} \label{thm:dis} Fix  ${\bf c} \in \Z^n_{>0}$. Given a graph $G$ on the vertex set $[n+1]$ with $m$ edges, the  reduction tree $R_G^{\bf c}$ of  $\Gc$ has  $$\sum_{\substack{{\bf j} \geq \lout_G \\ j_1 + \dots + j_n = m-n}} K_G(j_1 -  \lout_{G}(1), \dots, j_n -  \lout_{G}(n), 0)$$  leaves, where: 
\begin{itemize}
\item  the sum is over  weak
compositions ${\bf j}=(j_1,j_2,\ldots,j_n)$ of $
m-n$ that are $\geq (\om_1,\ldots,\om_n)$ in dominance order;

\item    $K_G(j_1 -  \lout_{G}(1), \dots, j_n -  \lout_{G}(n), 0)$ of the leaves of $R_G^{\bf c}$ are   ${G[{\bf j +1}]({\bf c})}$;
\item   the polytopes $\mathcal F_{G[{\bf j +1}]({\bf c})}({e_1-e_{n+2}})$   are  interior disjoint and their union is $\mathcal F_{\Gc}(e_1-e_{n+2})$;

\item   the polytopes $\mathcal F_{G[{\bf j +1}]({\bf c})}({e_1-e_{n+2}})$   are of the same dimension as $\mathcal F_{\Gc}(e_1-e_{n+2})$.
\end{itemize}
\end{theorem}

Before proceeding with the proof of Theorem \ref{thm:dis}, we illustrate it with an example. Compare this to Example \ref{ex:1}.

\begin{example} \label{ex:2} Let us apply Theorem \ref{thm:dis} to the reduction tree $R_{K_4}^{(3,2,2)}$ from Figure~\ref{fig:complete}. Here $n=3$ and $m=6$. 
We have that $(\outdeg_{K_4}(1), \outdeg_{K_4}(2), \outdeg_{K_4}(3))=(3,2,1)$ and thus, $(\lout_{K_4}(1), \lout_{K_4}(2), \lout_{K_4}(3))=(2,1,0)$. The sum in Theorem  \ref{thm:dis}  is over  compositions ${\bf j}=(j_1,j_2,j_3)$ of $
3$ that are $\geq (2,1,0)$ in dominance order. Thus, the two possible ${\bf j}$'s are $(3,0,0)$ and $(2,1,0)$. The first two points of Theorem  \ref{thm:dis} thus state that there are $K_{K_4}(3 - 2, 0-1, 0-0, 0)=1$   leaves of $R_{K_4}^{(3,2,2)}$ that  are   ${G[4,1,1]}(3,2,2)$ and  $K_{K_4}(2 - 2, 1-1, 0-0, 0)=1$   leaves of $R_{K_4}^{(3,2,2)}$ that  are   ${G[3,2,1]}(3,2,2)$; see Figure~\ref{fig:complete}. 

\end{example}

\noindent {\it Proof of Theorem \ref{thm:dis}} Definitions \ref{def:rg1} and \ref{def:rg2} are set up so that appealing to Lemma \ref{lem:sub} and Theorem \ref{thm:a} instantly implies the first three statements of Theorem \ref{thm:dis}.   It remains to show that  the polytopes $\mathcal F_{G[{\bf j +1}]({\bf c})}({e_1-e_{n+2}})$    are of the same dimension as $\mathcal F_{\Gc}(e_1-e_{n+2})$. Since the dimension of  $\mathcal F_{H}(e_1-e_{n+2})$, where $H$ is a graph on the vertex set $[0, n+1]$ is  $|E(H)|-|V(H)|+1$, the same dimensionality of $\mathcal F_{G[{\bf j +1}]({\bf c})}({e_1-e_{n+2}})$ and  $\mathcal F_{\Gc}(e_1-e_{n+2})$ readily follows for ${\bf j} \geq \lout_G,  j_1 + \dots + j_n = m-n.$ \qed

\begin{lemma}\label{lem:dis} Fix  ${\bf c} \in \Z^n_{>0}$.  There is a dissection of $\mathcal F_{G[{\bf j +1}]({\bf c})}(e_1 - e_{n+2})$ into $\frac{(c_1)_{j_1}}{j_1!}\dots\frac{(c_n)_{j_n}}{j_n!}$ many unimodular simplices.
\end{lemma} 

Before giving a proof of Lemma \ref{lem:dis} we include a version of the Subdivision Lemma appearing in  \cite[Lemma 3.4]{MMlidskii}.

\begin{lemma} \label{lem:??}  \cite[Lemma 3.4]{MMlidskii} Let  $\{\F_{G^{(i)}_T(\mathcal{I}_i, \mathcal{O}_i)}(\g)\}_{T \in \mathcal{T}_{\mathcal{I}_i \cup \{i\}, \mathcal{O}_i}}$ as in Lemma \ref{lem:sub}. Let  $\mathcal{I}_i, \mathcal{O}_i$ be the multiset of incoming and outgoing edges incident to $i$, with fixed arbitrary ordering. Assume that $a_i=0$. Then, it is exactly those of the polytopes  among $\{\F_{G^{(i)}_T(\mathcal{I}_i, \mathcal{O}_i)}(\g)\}_{T \in \mathcal{T}_{\mathcal{I}_i \cup \{i\}, \mathcal{O}_i}}$ that  are of the same dimension as $\F_{G}(\g)$ for which  there is exactly one edge incident to $i$ in $T \in  \mathcal{T}_{\mathcal{I}_i \cup \{i\}, \mathcal{O}_i}$. Such polytopes form a dissection of  $\F_{G}(\g)$.
\end{lemma}

\noindent {\it Proof of Lemma \ref{lem:dis}.}  Repeatedly use Lemma \ref{lem:??} for the case of netflow vector coordinate value $0$   on the vertices $1,2, \ldots, n$ of $G[{\bf j +1}]({\bf c})$. This amounts to picking tuples of bipartite noncrossing trees $(T_1,\ldots, T_n) \in \mathcal{T}_{L_1,R_1}\times \cdots \times \mathcal{T}_{L_n,R_n}$, where $|L_i|=c_i$ and $|R_i|=j_i+1$, for $i \in [n]$.  The number of such tuples is $\frac{(c_1)_{j_1}}{j_1!}\dots\frac{(c_n)_{j_n}}{j_n!},$ since $\frac{(c_i)_{j_i}}{j_i!} =\#
\mathcal{T}_{L_i,R_i}$.\qed

\begin{theorem}\label{dis:simp} Fix  ${\bf c} \in \Z^n_{>0}$.  There is a dissection of $\mathcal F_{\Gc}(e_1 - e_{n+2})$ into  $$\sum_{\substack{{\bf j} \geq \lout_G \\ j_1 + \dots + j_n = m-n}} \frac{(c_1)_{j_1}}{j_1!}\dots\frac{(c_n)_{j_n}}{j_n!} K_G(j_1 -  \lout_{G}(1), \dots, j_n -  \lout_{G}(n), 0)$$ many unimodular simplices. \end{theorem}
\proof Follows readily from Theorem \ref{thm:dis} and Lemma \ref{lem:dis}.\qed

\section{The geometric proof of   \eqref{eq:kost}}
\label{sec:geom}

In this section we prove the Baldoni--Vergne--Lidskii integer point formula \eqref{eq:kost}  from Theorem~\ref{thmA}. As mentioned in the Introduction, the original proof by Baldoni and Vergne \cite{bv}  relies on residue calculations and a second, combinatorial  proof by M\'esz\'aros and Morales \cite{MMlidskii}  makes use of a canonical subdivision of flow polytopes and   generating functions of Kostant partition functions to prove \eqref{eq:kost}. In contrast, here we give a purely geometric proof of \eqref{eq:kost}.  For the reader's reference we rewrite  \eqref{eq:kost}  in Theorem \ref{thm:main-sec-4}  in the form that we prove it:

\begin{theorem}
\label{thm:main-sec-4}
Let $G$ be a connected graph on vertex set $[n+1]$ so that $G$ has at least one outgoing edge at vertex $i$ for $i \leq n$. For $i \in [n]$ we set  $ \inn_{G}(i) \colonequals \indeg_G(i) - 1$ and $ \lout_{G}(i)\colonequals \outdeg_G(i) - 1$. Fix positive integers $c_1, \dots, c_n$. Let $a_i\colonequals  \inn_{G}(i) + c_i$. Then we have
$$K_G\B(a_1, \dots, a_n, -\sum_{i=1}^na_i\B) = \sum_{\substack{{\bf j} \geq \lout_G \\ j_1 + \dots + j_n = m-n}} \frac{(c_1)_{j_1}}{j_1!}\dots\frac{(c_n)_{j_n}}{j_n!} K_G(j_1 -  \lout_{G}(1), \dots, j_n -  \lout_{G}(n), 0),$$
where $\geq$ denotes the dominance order, that is, $j_1 + \dots + j_k \geq  \lout_{G}(1) + \dots + (\lout_G)_k$ for all $k \in [n]$, and $(n)_k \colonequals n(n+1)\dots(n+k-1)$.
\end{theorem}

\proof   Given a vector ${\bf c}:=(c_1, \dots, c_n) \in \Z^n_{>0}$  and a graph $G$ on the vertex set $[n+1]$, we defined  $G(\bf c)$ on the vertex set $[0,n+1]$ so that 
$$K_G\B(a_1, \dots, a_n, -\sum_{i=1}^na_i\B) = K_{\Gc}\B(0,  \inn_{\Gc}(1), \dots,  \inn_{\Gc}(n), -\sum_{i=1}^n  \inn_{\Gc}(i)\B),$$
where $a_i \colonequals  \inn_{G}(i) + c_i$. See Definition \ref{def:c} and Lemma \ref{lem:in-vector}. 

  By \eqref{eq:vol}  the normalized volume of $\mathcal F_{\Gc}(e_1 - e_{n+2})$ is precisely $$K_{\Gc}(0,  \inn_{\Gc}(1), \dots,  \inn_{\Gc}(n), -\sum_{i=1}^n  \inn_{\Gc}(i)).$$ (Recall that the proof of  \eqref{eq:vol} given in \cite{MMlidskii} via the canonical subdivision is fully geometric.) In particular, the number of simplices in a unimodular triangulation of $\mathcal F_{\Gc}(e_1 - e_{n+2})$ is $$K_{\Gc}(0,  \inn_{\Gc}(1), \dots,  \inn_{\Gc}(n), -\sum_{i=1}^n  \inn_{\Gc}(i)).$$

By Theorem \ref{dis:simp} there is a dissection  of  $\mathcal F_{\Gc}(e_1 - e_{n+2})$ into $$\sum_{\substack{{\bf j} \geq \lout_G \\ j_1 + \dots + j_n = m-n}} \frac{(c_1)_{j_1}}{j_1!}\dots\frac{(c_n)_{j_n}}{j_n!} K_G(j_1 -  \lout_{G}(1), \dots, j_n -  \lout_{G}(n), 0)$$ many unimodular simplices. 
\bigskip

Thus, chaining all the equalities we get that
\begin{align*}
K_G\B(a_1, \dots, a_n, -\sum_{i=1}^na_i\B) &= K_{\Gc}\B(0,  \inn_{\Gc}(1), \dots,  \inn_{\Gc}(n), -\sum_{i=1}^n  \inn_{\Gc}(i)\B) \\&=\sum_{\substack{{\bf j} \geq \lout_G \\ j_1 + \dots + j_n = m-n}} \frac{(c_1)_{j_1}}{j_1!}\dots\frac{(c_n)_{j_n}}{j_n!} K_G(j_1 -  \lout_{G}(1), \dots, j_n -  \lout_{G}(n), 0),
\end{align*}
to  obtain    Theorem ~\ref{thm:main-sec-4}. 
\qed

\section{Concluding remarks}
\label{sec:con}

Faced with the formulas for volume and integer point count of flow polytopes given  in  equations \eqref{eq:vol} and \eqref{eq:kost}, one instantly observes the nonnegativity of the quantities involved. Yet, the original proof of Baldoni and Vergne \cite{bv} is via residue calculations: involving complex numbers and subtractions.  

When we study manifestly nonnegative quantities, as in equations \eqref{eq:vol} and \eqref{eq:kost}, it is natural to seek a manifestly nonnegative proof: a proof devoid of subtraction (and complex numbers). A geometric proof can make this aspiration a reality. A  geometric construction was used  by  M\'esz\'aros and   
 Morales  \cite{MMlidskii}  to prove \eqref{eq:vol} in a manifestly nonnegative way and the present paper accomplishes the same goal via geometric constructions for~\eqref{eq:kost}.   
 
 \section*{Acknowledgments} 
		We are grateful to Lou Billera for inspiring conversations. We thank the Institute for Advanced Study for providing a hospitable environment for our collaboration. The first and third authors also thank the Einhorn Discovery Grant and the Cornell Mathematics Department for providing the funding for their visits to the Institute for Advanced Study.  

\bibliography{integer-points-of-flow-polytopes-arxiv}

\begin{thebibliography}{10}

\bibitem{bv}
W.~Baldoni and M.~Vergne.
\newblock {K}ostant partitions functions and flow polytopes.
\newblock {\em Transform. Groups}, 13(3-4):447--469, 2008.

\bibitem{counting}
W.~Baldoni-Silva, J.~A. De~Loera, and M.~Vergne.
\newblock Counting integer flows in networks.
\newblock {\em Found. Comput. Math.}, 4(3):277--314, 2004.

\bibitem{deloerasturm}
J.~A. De~Loera and B.~Sturmfels.
\newblock Algebraic unimodular counting.
\newblock {\em Math. Program.}, 96(2, Ser. B):183--203, 2003.
\newblock Algebraic and geometric methods in discrete optimization.

\bibitem{GMP}
S.~C.. {Gutekunst}, K.~{M\'esz\'aros}, and T.~K. {Petersen}.
\newblock {Root cones and the resonance arrangement}.
\newblock {\em ArXiv e-prints}, 2019.

\bibitem{ubiquity}
A.~N. Kirillov.
\newblock Ubiquity of {K}ostka polynomials.
\newblock In {\em Physics and combinatorics 1999 ({N}agoya)}, pages 85--200.
  World Sci. Publ., River Edge, NJ, 2001.

\bibitem{kostant2}
B.~Kostant.
\newblock A formula for the multiplicity of a weight.
\newblock {\em Proc. Nat. Acad. Sci. U.S.A.}, 44:588--589, 1958.

\bibitem{kostant1}
B.~Kostant.
\newblock A formula for the multiplicity of a weight.
\newblock {\em Trans. Amer. Math. Soc.}, 93:53--73, 1959.

\bibitem{diag}
R.~I. Liu, K.~M\'esz\'aros, and A.~H. Morales.
\newblock Flow polytopes and the space of diagonal harmonics.
\newblock {\em Canadian Journal of Mathematics}, 71(6):1495--1521, 2019.

\bibitem{mm}
K.~M\'{e}sz\'{a}ros and A.~H. Morales.
\newblock Flow polytopes of signed graphs and the {K}ostant partition function.
\newblock {\em Int. Math. Res. Not. IMRN}, (3):830--871, 2015.

\bibitem{MMlidskii}
K.~M\'esz\'aros and A.~H. Morales.
\newblock {Volumes and Ehrhart polynomials of flow polytopes}.
\newblock {\em Math. Z.,}, (293):1369--1401, 2019.

\bibitem{tesler}
K.~M\'{e}sz\'{a}ros, A.~H. Morales, and B.~Rhoades.
\newblock The polytope of {T}esler matrices.
\newblock {\em Selecta Math. (N.S.)}, 23(1):425--454, 2017.

\bibitem{groth}
K.~M\'esz\'aros and A.~St.~Dizier.
\newblock From generalized permutahedra to grothendieck polynomials via flow
  polytopes.
\newblock {\em Algebraic Combinatorics}, 3(5):1197--1229, 2020.

\bibitem{schrijver}
A~Schrijver.
\newblock {\em Combinatorial optimization. {P}olyhedra and efficiency. {V}ol.
  {B}}, volume~24 of {\em Algorithms and Combinatorics}.
\newblock Springer-Verlag, Berlin, 2003.
\newblock Matroids, trees, stable sets, Chapters 39--69.

\bibitem{sturm}
B.~Sturmfels.
\newblock On vector partition functions.
\newblock {\em J. Combin. Theory Ser. A}, 72(2):302--309, 1995.

\end{thebibliography}
\bibliographystyle{plain}

\end{document}